\documentclass[reqno, 11pt]{article}

\usepackage{amsmath}
\usepackage{amsthm}
\usepackage{amsfonts}
\usepackage{amssymb}
\usepackage{url}
\usepackage{enumerate}
\usepackage[pdftex,bookmarks=true]{hyperref}
\usepackage{xcolor}
\usepackage{graphicx}

\newcommand\R{{\mathbb{R}}}
\newcommand\C{{\mathbb{C}}}

\renewcommand\P{{\mathbb{P}}}
\newcommand\E{{\mathbb{E}}}

\newcommand\Var{\mathbf{Var}}
\newcommand\eps{{\varepsilon}}

\newcommand\dist{\operatorname{dist}}

\theoremstyle{plain}
 \newtheorem{theorem}{Theorem}[section]
 \newtheorem{conjecture}[theorem]{Conjecture}
 
 \newtheorem{problem}[theorem]{Problem}

 \newtheorem{proposition}[theorem]{Proposition}
 
 \newtheorem{lemma}[theorem]{Lemma}

\newtheorem{remark}[theorem]{Remark}

\theoremstyle{definition}
\newtheorem{definition}[theorem]{Definition}

\title{Complex Random Matrices have no Real Eigenvalues}

\author{Kyle Luh \thanks{Department of Mathematics, Yale University. Email:
		kyle.luh@yale.edu and Harvard University.  Email: kluh@seas.harvard.edu}
			}
\date{}

\begin{document}
\maketitle
\begin{abstract}
	Let $\zeta = \xi + i\xi'$ where $\xi, \xi'$ are iid copies of a mean zero, variance one, subgaussian random variable.  Let $N_n$ be a $n \times n$ random matrix with entries that are iid copies of $\zeta$.  We prove that there exists a $c \in (0,1)$ such that the probability that $N_n$ has any real eigenvalues is less than $c^n$ where $c$ only depends on the subgaussian moment of $\xi$.  The bound is optimal up to the value of the constant $c$.  The principal component of the proof is an optimal tail bound on the least singular value of matrices of the form $M_n := M + N_n$ where $M$ is a deterministic complex matrix with the condition that $\|M\| \leq K n^{1/2}$ for some constant $K$ depending on the subgaussian moment of $\xi$.  For this class of random variables, this result improves on the results of Pan-Zhou \cite{pan2010circular} and Rudelson-Vershynin \cite{rudelson2008littlewood}.  In the proof of the tail bound, we develop an optimal small-ball probability bound for complex random variables that generalizes the Littlewood-Offord theory developed by Tao-Vu (\cite{tao2009littlewood,tao2009inverse}) and Rudelson-Vershynin (\cite{rudelson2008littlewood,rudelson2009smallest}).   
\end{abstract}
\section{Introduction}
The study of eigenvalues is a foundational aspect of random matrix theory.  For non-symmetric random matrices, real eigenvalues are of particular interest  (\cite{girko2012theory,edelman1994many,edelman1997probability,lehmann1991eigenvalue,hameed2015real,forrester2016real}).   However, very little is known about their behavior for general classes of random variables \cite[Ch. 15.3]{mehta2004random}.  Before a recent universality result of Tao and Vu \cite{taovuuniversality}, even the existence of real eigenvalues was only known for the real gaussian case, where it has been proven that there are roughly $\sqrt{\frac{2n}{\pi}} $ real eigenvalues \cite{edelman1994many}.  Tao and Vu extended this result to all random variables that match the real gaussian up to the first four moments \cite{taovuuniversality}.  Yet, even this result does not shed light on many natural distributions (e.g. Rademacher $\pm 1$).  We demonstrate that for a general class of complex random variables, whose real and imaginary components are independent, there are unlikely to be any real eigenvalues.  

A key element of the proof, which is of independent interest, is an optimal result on the tail probability of the least singular value of random complex matrices.  Let $M$ be a $n \times n$ matrix and $s_1(M) \geq \dots \geq s_n(M)$ its singular values.  Of great interest to numerical analysts is the condition number of the matrix $M$, defined to be 
$$
\kappa(M) := s_1(M)/s_n(M) = \|M\|  \|M^{-1}\|.
$$
Attempts to understand the typical behavior of this parameter were instigated by von Neumann and Goldstine \cite{von1947numerical} in their seminal work on numerical matrix inversion.  The condition number is also intimately tied to the efficiency of algorithms \cite{smale1985efficiency} and the difficulty of problems in numerical analysis \cite{demmel1988probability}. Spielman and Teng were motivated by these concerns when they introduced the paradigm of smoothed analysis \cite{spielman2004smoothed, spielman2009smoothed}.  Their goal was to understand the behavior of algorithms on fixed inputs that had been perturbed by random noise.  Since the operator norm of a random matrix is well-understood, the difficulty in the analysis of the condition number reduces to understanding the least singular value.  In our setting, we examine the least singular value of a fixed matrix $M$ plus a random matrix $N_n$.  The real case with Gaussian noise was addressed in \cite{sankar2006smoothed} and more general models were featured in \cite{tao2010smooth}.  There is a compelling practical motivation for understanding these more general models because discrete models, in particular, are more accurate representations of noise and error in digital settings.  

\section{Previous Results}     
\subsection{Real Eigenvalues}
Edelman, Kostlan, and Shub were able to find precise asymptotics for the expected number of real eigenvalues, $\E E_n$, for a $n \times n$ random matrix with iid $\mathcal{N}(0,1)$ entries, which we will refer to as the Gaussian ensemble.  
\begin{theorem}[\cite{edelman1994many}]
For the $n \times n$ Gaussian ensemble,
	$$
	\lim_{n \to \infty} \frac{\E E_n}{\sqrt n} = \sqrt{\frac{2}{\pi}}.
	$$
\end{theorem}
Later, Forrester and Nagao were able to control the variance of this statistic.
\begin{theorem}[\cite{forrester2007eigenvalue}]
For the $n \times n$ Gaussian ensemble,
$$
\Var(E_n) = (2-\sqrt 2) \sqrt{\frac{2n}{\pi}} + o(\sqrt n) .
$$
\end{theorem}
Edelman \cite{edelman1997probability} also derived exact formulas for the probability that the random real gaussian matrix has exactly $k$ real eigenvalues and expressed the joint densities of those eigenvalues explicitly.  The following theorem is a consequence of evaluating the formula for $k=n$.
\begin{theorem}[\cite{edelman1997probability}]
	The probability that the random real gaussian matrix has all real eigenvalues is $2^{-n(n-1)/4}$.
\end{theorem}
The techniques used in the proof of the above results are specialized for gaussian random variables. Recently, Tao and Vu were able to extend this result to a larger class of random variables.

\begin{theorem}[\cite{taovuuniversality}, Corollary 17]
For a real random matrix $N_n$ with entries $\xi_{ij}$ such that
$$
\P(|\xi_{ij}| \geq t) \leq C \exp(-t^c)
$$
for constants $C, c >0$ (independent of $n$) for all $i,j$ and $\xi_{ij}$ match the moments of $\mathcal{N}(0,1)$, then 
$$
\E E_n = \sqrt{\frac{2n}{\pi}} + O(n^{1/2-c'})
$$
and
$$
\Var E_n = O(n^{1-c'})
$$
for some fixed $c' >0$.
\end{theorem}
 Outside of this class of random variables, almost nothing is understood.  In fact, the following toy problem was posed in Van Vu's talk at the 2014 ICM in Seoul and remains unresolved.
\begin{problem}
	Prove that a random $\pm 1$ random matrix has at least two real eigenvalues with high probability.
\end{problem}
\subsection{Least Singular Value}
In contrast to real eigenvalues, much is known about the universality of the least singular value in the real case.  One can deduce from a result of Edelman \cite{edelman1988eigenvalues} that
\begin{theorem}
	For $N_n$, a random matrix populated with iid $\mathcal{N}(0,1)$ random variables, we have that for any $\eps > 0$,
	$$
	\P(s_n(N_n) \leq \eps n^{-1/2}) \leq \eps.
	$$
\end{theorem}   
Sankar, Teng, and Spielman \cite{sankar2006smoothed} were able to prove an analogous result for the smoothed analysis model.
\begin{theorem}
	There exists a constant $C >0$ such that for $M$, a deterministic matrix, and $N_n$ a random matrix populated with iid $\mathcal{N}(0,1)$ random variables, we have that for any $\eps > 0$ and $M_n := M + N_n$,
	$$
	\P(s_n(M_n) \leq \eps n^{-1/2})\leq C \eps.
	$$
\end{theorem}
They further conjectured that
\begin{conjecture}
Let $\xi$ be a mean zero, variance at least 1, subgaussian random variable.  Let $N_n$ be a $n \times n$ random matrix with iid entries $\xi$.  There exists constants $C, c > 0$ such that for every $\eps \geq 0$
	$$
	\P(s_n(N_n) \leq \eps n^{-1/2}) \leq C \eps + c^n.
	$$ 
\end{conjecture}
\begin{definition}
	A random variable $\xi$ is \emph{subgaussian} if there exists a $B > 0$ such that
	$$
	\P(|\xi| > t) \leq 2 \exp( -t^2/B^2) \text{ for all } t>0.
	$$
	The minimal $B$ in the inequality is known as the \emph{subgaussian moment} of $\xi$.
\end{definition}
For a very general class of random variables, Tao and Vu \cite{tao2010smooth} showed that 
\begin{theorem}
Let $\xi$ be a random variable with mean zero and bounded second moment, and let $\gamma\geq 1/2$, $A\geq 0$ be constants.  Then there is a constant $C$ depending on $\xi,\gamma, A$ such that the following holds.  Let $N_n$ be the random matrix of size $n$ whose entries are iid copies of $\xi$.  Let $M$ be a deterministic matrix satisfying $\|M\| \leq n^{\gamma}$ and let $M_n :=M + N_n$.  Then
$$
\P(s_n(M_n) \leq n^{-(2A+2)\gamma + 1/2}) \leq C \left(n^{-A+ o(1)} + \P(\|N_n\| \geq n^{\gamma}) \right).
$$
\end{theorem}
Furthermore, they showed that unlike the gaussian case, the bound necessarily requires conditions on $M$.  
In \cite{rudelson2008littlewood}, Rudelson and Vershynin obtained the optimal rate for subgaussian random variables and $M = 0$.

\begin{theorem}
	Let $\xi_1, \dots, \xi_n$ be independment random variables with mean zero, variance at least 1, and subgaussian moments bounded by $B$.  Let $N_n$ be a $n \times n$ random matrix whose rows are independent copies of the random vector $(\xi_1, \dots, \xi_n)$.  Then for every $\eps \geq 0$ one has
	$$
	\P(s_n(N_n) \leq \eps n^{-1/2}) \leq C \eps + c^n
	$$ 
	where $C > 0$ and $c \in (0,1)$ depend only on $B$.
\end{theorem} 
For the complex case, Edelman's work \cite{edelman1988eigenvalues} implies the following.
\begin{theorem}
	For $\zeta$ a complex gaussian and $N_n$ a random $n \times n$ matrix populated with iid entries $\zeta$, then for all $\eps \geq 0$
	$$
	\P(s_n(N_n) \leq \eps n^{-1/2}) \leq \eps^2.
	$$
\end{theorem}
For more general complex random variables, Pan and Zhou \cite{pan2010circular}, modifying the argument of Rudelson and Vershynin \cite{rudelson2008littlewood}, showed
\begin{theorem}
	Let $\zeta$ be a complex random variables with mean zero, $\mathbb{E}|\zeta|^2 = 1$, and $\mathbb{E}|\zeta|^3 < B$.  Let $M$ be a fixed complex matrix and $N_n$ be a random matrix with iid entries $\zeta$ and define $M_n := M + N_n$.  There exists a $C >0$ and $c \in (0,1)$ such that for $K \geq 1$ and every $\eps > 0$,
	$$
	\P(s_n(M_n) \leq \eps n^{-1/2}) \leq C \eps + c^n + \P(\|M_n\| > K \sqrt n)
	$$
	where $C, c$ only depend on $K, B, \E (Re(\zeta))^2, \E(Im(\zeta))^2,$ and $\E Re(\zeta)Im(\zeta)$.
\end{theorem}
Our work will improve this rate for the case when the real and imaginary components of the random variable are independent. 

\section{Main Results}
We address the question of the existence of real eigenvalues for a general class of complex random variables whose real and imaginary parts are independent.

\begin{definition}
We say that a random variable $\zeta$ is \emph{genuinely complex with moment $B$} if $\zeta = \xi + i \xi'$ where $\xi$ and $\xi'$ are iid, mean zero, variance 1 and subgaussian with moment $B$.  
\end{definition}

\begin{theorem}\label{MainNoReal}
	Let $N_n$ be a $n \times n$ random matrix populated with independent copies of a random variable that is genuinely complex with moment $B$.  Then there exists a $c_{\ref{MainNoReal}} \in (0,1)$ only depending on $B$ such that
	$$
	\P(N_n \text{ has a real eigenvalue}) \leq c_{\ref{MainNoReal}}^n.
	$$
\end{theorem}  

\begin{remark}
	This is best possible up to the value of the constant $c_{\ref{MainNoReal}}$.  For example, for $\pm 1 \pm i$ random variables, the probability of having zero as an eigenvalue is lower bounded by the probability that there exists two rows or columns with the same entries.  The probability of the latter is $(1+o(1))n^2 4^{-n}$.
\end{remark}

The crucial ingredient in the proof is a new result on the smoothed analysis of the least singular value for such complex matrices.
\begin{theorem}\label{LeastSingular}
	Let $N_n$ be as in Theorem \ref{MainNoReal} and let $K > 0$ be a constant.  There exists constants $C_{\ref{LeastSingular}}, c_{\ref{LeastSingular}}>0$ only depending on $B, K$ such that for $M$, a fixed complex matrix with $\|M\| \leq K \sqrt n$, $M_n := M + N_n$, and for all $\eps \geq 0$
	$$
	\P(s_n(M_n) \leq \eps n^{-1/2}) \leq C_{\ref{LeastSingular}} \eps^2 + c_{\ref{LeastSingular}}^n.
	$$    
\end{theorem} 
\begin{remark}
	Edelman's result \cite{edelman1988eigenvalues} shows that $C \eps^2$ is optimal up to the constant $C$.  Setting $\eps = 0$ and considering $\pm 1$ random variables, we recover the complex analogue of the Kahn, Koml{\'o}s, and Szemeredi \cite{kahn1995probability} result that $\pm 1$ random matrices are singular with exponentially small probability. Thus, the $c^n$ term is optimal for random sign matrices.
\end{remark}
\section{Notation}

It will often be convenient to convert a problem from the complex setting to the real one.  For this purpose, we introduce the following notation.  For $v = (v_1, \dots, v_n)^T \in \C^n$ (all vectors are assumed to be column vectors), we let $\hat{v} := \left(\Re(v_1), \dots, \Re(v_n), \Im(v_1), \dots , \Im(v_n)\right)^T \in \R^{2n}$ where $\Re(v_j)$ and $\Im(v_j)$ are respectively the real and imaginary parts of the complex number $v_j$.  We will also need to convert $v$ into matrix form.  Let $[v] \in \R^{2 \times 2n}$ be defined as 
\begin{equation*}
 [v]:=  
 \left( \begin{array}{cc}
 \Re(v)^T & -\Im(v)^T  \\
 \Im(v)^T & \Re(v)^T \end{array} \right)
\end{equation*}         
where $\Re(v)$ indicates the vector whose entries are the real parts of the corresponding entries in $v$.  $\Im(v)$ is similarly defined (See Figure \ref{notation}).
\begin{figure}
\begin{center}
	\includegraphics[scale=.75]{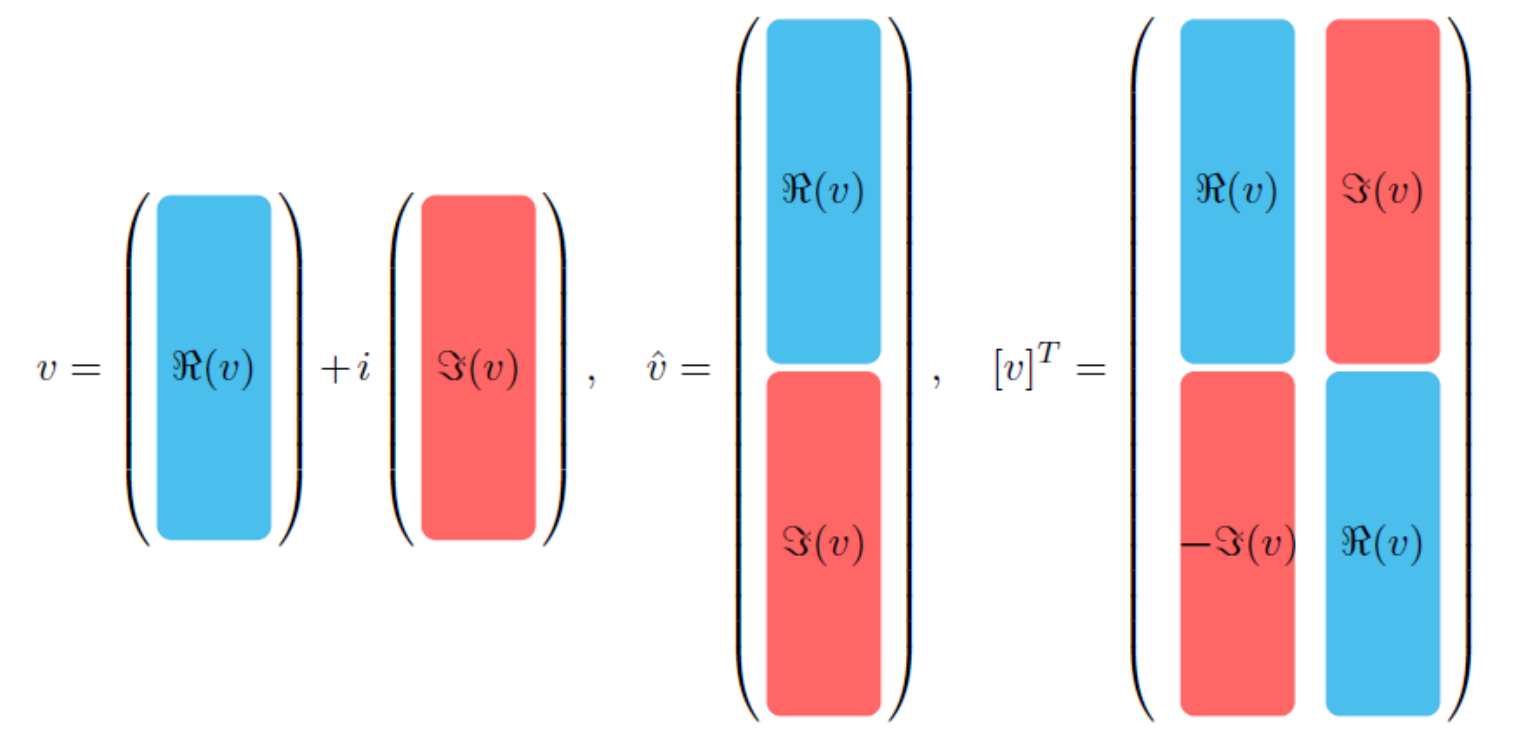} 
	\caption{Operations on $v$. (Note the added minus sign in the definition of $[v]$)}
	\label{notation}
\end{center}
\end{figure}

An important property is that for $a \in \C^n$, 
$$
\left|\sum_{j=1}^n a_j v_j \right| = |v^T a| = \| [v] \hat{a}\|_2 
$$

We use $\mathcal{S}_{\R}^{n-1}$ and $\mathcal{S}_{\C}^{n-1}$ to denote the unit sphere in $\R^n$ and $\C^n$ respectively.  For an $N \times N'$ real or complex matrix $M$, we denote the $\ell_2$ operator norm by $\|M\|$.  For $J \in [N]$ we denote by $M_J$ the $|J| \times N'$ matrix composed of the rows of $M$ indexed by $J$. For two vectors $v,v'$, let $v \cdot v'$ represent the standard dot product of the two.  $i$ will always mean $\sqrt{-1}$.  $\zeta$ will typically denote a complex random variable and $\xi$ a real one.  Additionally, $C$ with or without a subscript will typically denote a large constant that may depend on other parameters (e.g. $B$) and $c$ with or without a subscript will denote a small constant in $(0,1)$ that may also depend on other parameters (typically $B$).

\section{Proof of Theorem \ref{LeastSingular}}
The argument will be a modification of that used by Rudelson and Vershynin \cite{rudelson2008littlewood}.  We begin with a decomposition of the complex unit sphere.
\subsection{Decomposition of $\mathcal{S}_\C^{n-1}$}  
\begin{definition}
	Let $\delta, \rho \in (0,1)$ be two constants.  A vector $v \in \C^n$ is called sparse if $|supp(\hat{v})| \leq 2\delta n$.  A vector $v \in \mathcal{S}_{\C}^{n-1}$ is compressible if it is within Euclidean distance $\rho$ from the set of all sparse vectors.  A vector $v \in \mathcal{S}_{\C}^{n-1}$ is called incompressible if it is not compressible.  We will denote the set of sparse, compressible, and incompressible vectors by $Sparse(\delta), Comp(\delta, \rho), Incomp(\delta, \rho)$ respectively.
\end{definition}  

The least singular value problem can thus be divided into two subproblems.
$$
\P(s_n(M_n) \leq \eps n^{-1/2}) \leq \P(\inf_{x \in Comp(\delta, \rho)} \|M_n x\|_2 \leq \eps n^{-1/2}) + \P(\inf_{x \in Incomp(\delta, \rho)} \|M_n x\|_2 \leq \eps n^{-1/2}) .
$$
We exploit the different properties of compressible and incompressible vectors to solve the problem for each set in a distinct way.

\subsection{Compressible Vectors}
For compressible vectors, the bound is much stronger than we need and the argument is essentially the same as \cite{rudelson2008littlewood,pan2010circular}.    
\begin{lemma} \label{Compressible}
	For $M_n$ as in Theorem \ref{LeastSingular}, there exist $\delta, \rho, c_{\ref{Compressible}}, c'_{\ref{Compressible}} >0$ only depending on $B$ and $K$ such that
	\begin{equation} \label{CompressibleBound}
	\P(\inf_{v \in Comp(\delta, \rho)} \|M_n v\|_2 \leq c_{\ref{Compressible}} n^{1/2}) \leq \exp(-c'_{\ref{Compressible}} n).
	\end{equation}
\end{lemma}
\begin{proof}
	See \cite[Section 2.2]{pan2010circular}.
\end{proof}

\subsection{Incompressible Vectors}
For the remainder of the proof we fix a $\delta$ and $\rho$ such that (\ref{CompressibleBound}) holds.
For incompressible vectors, we leverage the fact that they have many coordinates of roughly the same size.  
\begin{lemma} \cite[Lemma 3.4]{rudelson2008littlewood} \label{spread}
	Let $z \in Incomp(\delta, \rho)$.  Then there exists a set $\sigma \subseteq \{1, \dots 2n\}$ of cardinality $|\sigma| \geq \nu_1 n$ and such that
	\begin{align*}\label{spread}
	\frac{\nu_2}{\sqrt n} &\leq |\hat{z}_k| \leq \frac{\nu_3}{\sqrt{n}} && \text{ for all } k \in \sigma 
	\end{align*}	
	where $0 < \nu_1, \nu_2, \nu_3$ are constants depending only on $\delta$ and $\rho$.  $\sigma$ is known as the \emph{spread part} of the vector $z$.
\end{lemma}

\subsubsection{Invertibility of Incompressible Vectors Via Distance}

\begin{lemma}[Invertibility via Distance]\cite[Lemma 5.6]{rudelson2014recent} \label{distance}
	Let $M_n$ be a complex random matrix.  Let $X_1, \dots, X_n$ denote the column vectors of $M_n$, and let $H_k$ denote the span of all the column vectors except the $k$-th.  Then for every $\delta, \rho \in (0,1)$ and every $\eps > 0$ one has
	$$
	\P(\inf_{z \in Incomp(\delta, \rho)} \|M_n z\|_2 < \eps \nu_2 n^{-1/2}) \leq \frac{1}{\nu_1 n} \sum_{k=1}^n \P(\dist(X_k, H_k) < \eps) 
	$$
	where $\nu_1, \nu_2$ are as in Lemma \ref{spread}.
\end{lemma}
\begin{remark}
The statement is slightly different from that of \cite{rudelson2014recent}.  A minor variation of the proof there gives the above result.
\end{remark}
This lemma reduces the invertibility issue into a distance problem.  As we will bound the maximum probability over all the distances and arbitrary shifts, it suffices to consider $dist(X_n, H_n)$ for concreteness.  In fact, conditioning on $H_n$, we can fix a unit normal vector independent of $X_n$.  The distance is then simply the norm of the dot product of this unit normal vector with an independent random vector, so the question becomes one of small ball probability.  

\subsubsection{Small Ball Probability}
For intuition and motivation, we briefly revert back to the real case.  Consider the linear combination, $S = \sum_{k=1}^n a_k \xi_k$.  
\begin{definition}
	The \emph{L{\'e}vy concentration function} of $S$ is defined as
	$$
	\mathcal{L}(S, \eps) := \sup_{v \in \R} \P(|S-v| \leq \eps ) .
	$$
\end{definition}
Clearly, the vector $a = (a_1, \dots, a_n)$ has a strong influence on the L{\'e}vy concentration.  For example, if 
$$
a = (\frac{1}{\sqrt 2}, \frac{1}{\sqrt 2}, 0,0,\dots, 0)
$$
and $\xi$ are independent Rademacher random variables ($\pm 1$), then $\P(S = 0) = \frac{1}{2}$.  On the other hand, if 
$$
a = (\frac{1}{\sqrt n}, \dots, \frac{1}{\sqrt n})
$$
then for even $n$, $\P(S = 0) = 2^{-n} {n \choose n/2}\sim n^{-1/2}$.  Littlewood and Offord \cite{littlewood1943number} initiated the study of the dependence of the L{\'e}vy function on the arithmetic structure of $a$.  Recently, Tao and Vu \cite{tao2009inverse} proposed that a large small-ball probability implies a strong additive structure.  Results in the classification of this additive structure are now called Inverse Littlewood-Offord theorems
\cite{tao2009inverse,tao2009littlewood,tao2012littlewood,tao2010sharp,rudelson2008littlewood,rudelson2009smallest,friedland2007bounds}.
We now introduce a two-dimensional small ball probability bound which corresponds to a bound on the Levy concentration for complex sums with complex coefficients.  Rudelson and Vershynin \cite{rudelson2008littlewood} proposed a measure for the additive strucutre of a vector $v \in \R^n$. They coined the term Essential Least Common Denominator (lcd).  
 $$
 lcd_{\alpha, \gamma}(v) := \inf \Big \{\hat{\theta} > 0: \dist (\hat{\theta} v , \mathbb{Z}^{n}) < \min(\gamma \|\hat{\theta} v\|_2, \alpha)\Big\}.
 $$  
We generalize this definition to handle complex vectors and our matrix construction $[v]$.

\begin{definition}
	Fix parameters $\gamma \in (0,1)$ and $\alpha > 0$,  
	we define the Essential Least Common Denominator of $v \in \C^n$ to be 
 $$
 LCD_{\alpha, \gamma}(v) := \inf \Big \{ \|\theta\|_2 : \theta \in \R^{2}, \dist([v]^T \theta, \mathbb{Z}^{2n}) < \min(\gamma \|\theta\|_2, \alpha)\Big \}.
 $$
\end{definition}
\begin{remark}
The anonymous referee has pointed out that such a generalization was introduced for matrix arguments with real entries (but not in the complex setting) by Rudelson and Vershynin in \cite{rudelson2009smallest} which renders a proof of Theorem \ref{smallball} unnecessary.
\end{remark}
We use this measure of structure to control the small-ball probability.
\begin{theorem}[Small Ball Probability via LCD]\label{smallball}
	Consider a random vector $\boldsymbol{\xi} = (\xi_1, \dots, \xi_{2n})$ with $\xi_j$ iid, mean 0, variance 1 and subgaussian with moment $B$, and a $v \in \mathcal{S}_{\C}^{n-1}$.  There exists constants $C_{\ref{smallball}}, c_{\ref{smallball}}$ depending only on $B$, such that for $\alpha > 0$ and for
	$$
	\eps \geq \frac{4}{LCD_{\alpha, \gamma}(v)}
	$$
	we have
	$$
	\sup_{w \in \R^2}\P(\|[v] \boldsymbol{\xi} - w\|_2 \leq \eps ) \leq C_{\ref{smallball}} \eps^2 + C_{\ref{smallball}}\exp(-c_{\ref{smallball}} \alpha^2).
	$$
\end{theorem}
\begin{remark}
	In our application, we will set $\alpha$ to be $\beta \sqrt n$ for some small constant $\beta$, so the $\exp(-c \alpha^2)$ term is negligible.
\end{remark}
\begin{proof}
	See the proof of \cite[Theorem 3.3]{rudelson2009smallest}.  We simply observe that our definition of $LCD$ coincides with that of \cite[Theorem 3.3]{rudelson2009smallest} and condition (3.2) of \cite[Theorem 3.3]{rudelson2009smallest} holds with equality since $v$ is a unit vector. 
    \end{proof}
    
    By the tensorization lemma \cite[Lemma 2.2]{rudelson2008littlewood}, we get the following bound for a single vector.
        \begin{lemma}[Invertibility for a Single Vector]\label{singlevector}
        	Let $M'_{mn}$ be a $m \times n$ complex random matrix with entries of the form $m_{ij} + \zeta_{ij}$ where $m_{ij}$ is a deterministic complex number and $\zeta_{ij}$ is genuinely complex with moment $B$.  Then for $\alpha>0$ and vector $v \in \mathcal{S}_\C^{n-1}$, and for every $\eps > 0$, satisfying
        	\begin{equation} \label{eq_epscondition}
        	\eps \geq \max \left(\frac{4}{LCD_{\alpha, \gamma}(v)}, \exp(- \alpha^2) \right)
        	\end{equation}
        	there exists a $C_{\ref{singlevector}}$ only depending on $B$ such that we have
        	$$
        	\P(\|M'_{nm} v\|_2 < \eps m^{1/2}) \leq (C_{\ref{singlevector}} \eps)^{2m}.
        	$$
        \end{lemma}

    \subsection{Random Normal Vectors have Large LCD}
    We now show that it is unlikely that a random normal vector will have small LCD by an $\eps$-net argument.
    We first prove a lower bound on the $LCD$ for incompressible vectors that will be of use in the proof of Lemma \ref{randomnormal}. 
        
        \begin{lemma} [Lower Bound on LCD]\label{lowerboundLCD}
        There exists constants $\gamma > 0$ and $\lambda > 0$ only depending on $\delta, \rho$ such that for any incompressible vector $v \in \mathcal{S}_\C^{n-1}$ and any $\alpha >0$ we have $LCD_{\alpha, \gamma}(v) \geq \lambda n^{1/2}$.	
        \end{lemma}

        \begin{proof}
        	Assume to the contrary that $LCD_{\alpha, \gamma}(v) < \lambda n^{1/2}$ where $\lambda$ will be specified later.  By definition of the $LCD$ there exists $\theta \in \mathbb{R}^2$ and $p \in \mathbb{Z}^{2n}$ such that 
        	\begin{equation}\label{LCDupper}
        	\|[v]^T \theta - p\|_2 < \gamma \|\theta\|_2 < \gamma \lambda n^{1/2}.
        	\end{equation}
        	
        	The key observation is that due to the symmetry of $[v]^T$ (See Figure \ref{notation}), if 
        	    	     \begin{equation*}
        	    	     \|[v]^T \theta - p\|_2 = \left\|[v]^T 
        	    	     \left(\begin{array}{c}
        	    	     \theta_1 \\
        	    	     \theta_2
        	    	     \end{array}\right) - \left(
        	    	     \begin{array}{c}
        	    	     
        	    	     p_1\\
        	    	     p_2 
        	    	     \end{array}\right) \right\|_2
        	    	     \end{equation*}
        	    	     where $p_1, p_2 \in \mathbb{Z}^n$
        	    	     then
        	    	     \begin{equation*}
        	    	     \|[v]^T \theta - p\|_2  = \|[v]^T \theta' - p'\|_2 := \left\|[v]^T 
        	    	     \left(\begin{array}{c}
        	    	     -\theta_2 \\
        	    	     \theta_1
        	    	     \end{array}\right) - \left(
        	    	     \begin{array}{c}
        	    	     
        	    	     -p_2\\

        	    	     p_1 \\
        	    	     
        	    	     \end{array}\right)\right \|_2.
        	    	     \end{equation*}
        	   Clearly, $\|\theta\|_2 = \|\theta'\|_2.$
        	
        	Recall the definition of the spread part of the vector $v$ from Lemma \ref{spread}.  Let $\sigma(v) \subseteq [2n]$ denote the spread part of the vector $v$.  Assume without loss of generality that half of the spread coordinates are real, i.e. $|\sigma(v) \cap [n]| > \frac{\nu_1}{2} n$. Fix a constant $k$ such that 
        	$$1/k^2 < \nu_1/4.$$
        	Since, $v$ is a unit vector, by Markov's inequality, there exists a set $I(v) \subset [n]$ of size at least $(1-1/k^2) n$ such that for $j \in I(v)$, $|\Im(v_j)| < k/ \sqrt n$.  Thus, 
        	\begin{equation}\label{smallset}
        	|\sigma(v) \cap [n] \cap I(v)| \geq \frac{\nu_1}{4} n.
        	\end{equation}
        	Now let 
        	$$J(v):=\left\{j\in [n] : \max\{|([v]^T \theta)_j - p_j|, |([v]^T \theta')_j + p_{j+n}|\} < \frac{ \sqrt 2 \gamma \lambda}{\sqrt{\nu_1}} \right\}.$$ 
        	We finally define 
        	$$
        	L(v) := \sigma(v) \cap [n]  \cap I(v) \cap J(v).
        	$$
    	    By equation (\ref{LCDupper}),we have that $|J(v)| \geq (1-\nu_1/8)n$ and combining this bound with equation (\ref{smallset}) yields
    	     $$
    	     |L(v)| \geq \frac{\nu_1}{8} n.
    	     $$

    	    To exploit the symmetry, we define 
    	    $$L'(v) := \{j: (j-n) \in L(v)\}.$$  
        	For any $j \in L(v)$, 
        	$$
        	|p_j| < |([v]^T \theta)_j| + \frac{\sqrt 2 \gamma \lambda}{\sqrt \nu_1} \leq |\theta_1| |\Re(v_j)| + |\theta_2| |\Im(v_j)| + \frac{\sqrt 2 \gamma \lambda}{\sqrt \nu_1} < (\nu_3 + k + \frac{\sqrt 2 \gamma}{\sqrt \nu_1}) \lambda < 1
        	$$   
        	for small enough $\lambda$ and assuming $\gamma < 1$.
        	Similarly, for $j \in L'(v)$
        	$$
        	|p_j| < |([v]^T \theta')_{j+n}| + \frac{\sqrt 2 \gamma \lambda}{\sqrt \nu_1} \leq |\theta_2| |\Re(v_{j+n})| + |\theta_1| |\Im(v_{j+n})| + \frac{\sqrt 2 \gamma \lambda}{\sqrt \nu_1} < (\nu_3 + k + \frac{\sqrt 2 \gamma}{\sqrt \nu_1}) \lambda < 1.
        	$$
        	Since $p_j$ must be an integer, this implies $p_j = 0$ for $j \in L(v) \cup L'(v)$.  
        	Thus, 
        	$$\|([v]^T \theta)_L - p_L\|_2 = \|([v]^T \theta)_L\|_2$$ and $$\|([v]^T \theta)_{L'} - p_{L'}\|_2 = \|([v]^T \theta)_{L'}\|_2.$$  Using the inequality $\sqrt{a^2 + b^2} \geq \frac{a+b}{\sqrt 2}$ which holds for $a,b \in \mathbb{R}$, we now lower bound $\|[v]^T \theta - p\|$ by 
        	$$
        	\frac{1}{\sqrt 2}\Big(\|([v]^T \theta)_L\|_2 + \|([v]^T \theta)_{L'}\|_2 \Big)=  \frac{1}{\sqrt 2}\Big(\|([v]^T \theta)_{L}\|_2 + \|([v]^T \theta')_{L}\|_2 \Big).
        	$$  
        	Note that by the definition of $I(v)$,
        	$$
        	\|\Im(v)_L \|_2 < k
        	$$
        	and by the definition of $\sigma(v)$,
        	$$
        	\|\Re(v)_L\|_2 \geq \frac{ \nu_2 \sqrt \nu_1}{2 \sqrt 2}.
        	$$
        	Now we have two cases to consider.  Let $c'>0$ be a constant such that
        	$$
        	\sqrt{1-c'^2} \frac{\nu_2 \sqrt \nu_1}{2 \sqrt 2}  - c'k> 0.
        	$$  
        	\begin{enumerate}
        		\item Assume that $|\theta_1| \geq c'\|\theta\|_2 $ and $|\theta_2| \geq c' \|\theta\|_2$.  In this case, and adding the condition that 
        		$$\theta_1 \Re(v)_L \cdot \theta_2\Im(v)_L  \geq 0
        		$$
        		we find that
        		\begin{align*}
        		\|([v]^T \theta)_L\|_2 &= \| \theta_1\Re(v)_L + \theta_2 \Im(v)_L\|_2   \\
        		&\geq \|\theta_1 \Re(v) \|_2 \\
        		&\geq c' \frac{\nu_2 \sqrt \nu_1}{2 \sqrt 2} \|\theta\|_2.
        		\end{align*}
        	If $\theta_1 \Re(v)_L \cdot \theta_2\Im(v)_L  < 0$ then 
        	$$
        	\|-\theta_2 \Re(v)_L + \theta_1\Im(v)_L \|_2 \geq |\theta_2|\|\Re(v)_L\|_2
        	$$	 
            so
            \begin{align*}
            \|([v]^T \theta')_L\|_2 &\geq \| -\theta_2\Re(v)_L + \theta_1 \Im(v)_L\|_2   \\
            &\geq c' \frac{\nu_2 \sqrt \nu_1}{2 \sqrt 2} \|\theta\|_2.
            \end{align*} 
            
            \item If we assume that $|\theta_2| < c'\|\theta\|_2$ then $|\theta_1| > \sqrt{1-c'^2} \|\theta\|_2$.  Therefore,
            $$
            \|([v]^T \theta)_L\|_2 \geq \Big| |\theta_1|\|\Re(v)_L\|_2 - |\theta_2| \|\Im(v)_L\|_2 \Big| \geq (\sqrt{1-c'^2} \frac{\nu_2 \sqrt \nu_1}{2 \sqrt 2}  - c'k) \|\theta\|_2 .
            $$
            By an identical arugment applied with $\theta'$, we obtain the same lower bound for $L'$ for the case $|\theta_1| < c'\|\theta\|_2$.  
        	\end{enumerate}
        	We have shown that
        	$$
        	\|[v]^T \theta- p\|_2 \geq \min\Big\{c' \frac{\nu_2 \sqrt \nu_1}{2 \sqrt 2}, \sqrt{1-c'^2} \frac{\nu_2 \sqrt \nu_1}{2 \sqrt 2}  - c'k\Big\} \|\theta\|_2.
        	$$
        	Setting $\gamma < \min\Big\{c' \frac{\nu_2 \sqrt \nu_1}{2 \sqrt 2}, \sqrt{1-c'^2} \frac{\nu_2 \sqrt \nu_1}{2 \sqrt 2}  - c'k\Big\}$ yields the desired contradiction.
        \end{proof}
        
       For the remainder of the proof, we fix $\lambda$ and $\gamma$ from Lemma \ref{lowerboundLCD}.  We divide the set of potential normal vectors into classes of similar $LCD$.
    \begin{definition}
    	Define $S_D = \{v \in \mathcal{S}_\C^{n-1} : D \leq LCD_{\alpha, \gamma}(x) \leq 2D \} $.
    \end{definition}
    
    \begin{lemma}[Nets for Level Sets for LCD]\label{net}
    	For some absolute constant $C_{\ref{net}} > 0$, there exists a $2\alpha/D$-net of $S_D$ of cardinality at most $ C_{\ref{net}} \frac{(2\alpha+2D)^2}{\alpha^2} \left(\frac{10D}{n^{1/2}}\right)^{2n}$ for $D \geq \lambda n^{1/2}$ and $\alpha = \beta n^{1/2}$ for any $\beta < \lambda$.
    \end{lemma}
    \begin{proof}
    	For a parameter $r$ to be chosen later, we create an $r$-net, $A_{D,r}$ of the annulus, $A_D$ in $\R^2$ defined by 
    	$$A_D := \{\theta: D \leq \|\theta\|_2  \leq 2D\}.$$ 
    	For every $v \in S_D$, there exists $\theta \in A_D$ and $p \in \R^{2n}$ such that 
    	$$\|[v]^T \theta - p\|_2 < \alpha.$$ 
    	Let $\theta' \in A_{D, r}$ be within $r$ of $\theta$.  For every $\theta'$ there  is a unique $v' \in \C^n$ such that $[v']^T \theta' = p$.  This can be seen by examining the $k$-th and $n+k$-th coordinates for all $1 \leq k \leq n$.  This reduces to the following set of linear equations.
    	\begin{equation*}
    	\left(
    	\begin{array}{cc}
    	\theta_1 & \theta_2 \\
    	\theta_2 & - \theta_1
    	\end{array}
    	\right)
    	\left(
    	\begin{array}{c}
    	\Re(v_k) \\
    	\Im(v_k)
    	\end{array}
    	\right) = 
    	\left(
    	\begin{array}{c}
    	p_k \\
    	p_{n+k}
    	\end{array}
    	\right).
    	\end{equation*} 
    	For any $\theta \neq 0$ the matrix
    	\begin{equation*}
    	\left(
    	\begin{array}{cc}
    	\theta_1 & \theta_2 \\
    	\theta_2 & - \theta_1
    	\end{array}
    	\right)
    	\end{equation*}
    	is invertible so the system has a unique solution $v'$.  The norm of $v'$ cannot be too large as
    		
    		\begin{equation*}
    		\left \|
    		\begin{array}{c}
    		\Re(v_k) \\
    		\Im(v_k)
    		\end{array}
    		\right\|_2 = \left \| \left(
    		\begin{array}{cc}
    		\theta_1 & \theta_2 \\
    		\theta_2 & - \theta_1
    		\end{array}
    		\right)^{-1} \right\| 
    		\left\|
    		\begin{array}{c}
    		p_k \\
    		p_{n+k}
    		\end{array}
    		\right \|_2
    		\leq \frac{1}{\|\theta\|_2} \left\|
    		\begin{array}{c}
    		p_k \\
    		p_{n+k}
    		\end{array}
    		\right \|_2
    		\end{equation*}
    	so $$\|v'\|_2 \leq \frac{1}{D} \|p\|_2 \leq \frac{1}{D} (\alpha + \|[v]^T \theta\|_2) \leq \frac{\alpha + 2D}{D}$$   
    	due to the orthogonality of the rows of $[v]$.  Also,
    	\begin{align*}
    	\|v - v'\|_2 &= \frac{1}{\|\theta\|_2} \|([v]^T - [v']^T) \theta\|_2   \\
    	             &\leq \frac{1}{ \|\theta\|_2} \left( \|[v]^T \theta - p\|_2 + \|[v']^T \theta' - p\|_2 + \|[v]'^T (\theta - \theta')\|_2 \right) \\
    	             &\leq \frac{1}{ \|\theta\|_2} \left( \|[v]^T \theta - p\|_2 + \|[v']^T \theta' - p\|_2 + \|v'\|_2 \|\theta - \theta'\|_2 \right) \\
    	             &\leq \frac{1}{D}(\alpha + r \frac{\alpha+ 2D}{D}) \\
    	             &\leq \frac{2 \alpha}{D}.
    	\end{align*}
    	The second inequality follows from the observation that the rows of $[v]-[v']$ are orthogonal and of the same length.  The last inequality is achieved by letting $r := \frac{D\alpha}{(\alpha+2D)}$.  Let $$\mathcal{N} = \{v \in \C^n : \exists \theta' \in A_{D,r} \text{ and } p \in \mathbb{Z}^{2n} \cap B(0,\alpha + 2D) \text{ such that } [v]^T \theta' = p \}.$$	
    	We have shown that $\mathcal{N}$ is an $2\alpha/D$-net of $S_D$.  Now we bound the cardinality of $\mathcal{N}$.
    	$$
    	|\mathcal{N}| \leq |A_{D, r}| \left(1+ \frac{3(\alpha+2D)}{\sqrt{2n}}\right)^{2n} \leq C_{\ref{net}} \frac{(2\alpha+2D)^2}{\alpha^2} \left(\frac{10D}{n^{1/2}}\right)^{2n}.
    	$$
    	These bounds follow from the well-known result on the number of lattice points in a high-dimensional sphere and a simple covering argument in the plane for the annulus.
    \end{proof}
    Now we use a basic covering argument and union bound to show that a vector orthogonal to $n-1$ rows of our random matrix is likely to have a large LCD.  We will first need a basic lemma on the operator norm of random matrices with subgaussian entries.  
    \begin{lemma}[\cite{rudelson2008littlewood}, Lemma 2.4] \label{lemma:operatornorm}
    For $N_n$ a random matrix with iid random variables which are genuinely complex with moment $B$, there exists a $K'>0$ only depending on $B$ such that
        $$
        \P(\|N_n\| > K' n^{1/2}) \leq 2 e^{-n}.
        $$ 
    \end{lemma}
    \begin{remark}
    In \cite{rudelson2008littlewood}, the statement of the lemma is only for real random variables, but splitting into real and imaginary components and then applying the triangle inequality yields the complex version.
    \end{remark}
    
	\begin{lemma}[Random Normal has Large LCD]\label{randomnormal}
		Let $N'_n$ be a $n-1 \times n$ random matrix with iid entries which are genuinely complex with moment $B$.  Denote by $M'_n = N'_n + M'$ where $M$ is a deterministic $(n-1) \times n$ matrix. Let $Z_1^T, \dots, Z_{n-1}^T$ designate the rows of $M'_n$.  Consider a vector $v$ orthogonal to all the $Z_j$.  Then there exists constants $c_{\ref{randomnormal}},c_{\ref{randomnormal}}'>0$ only depending on $B$ and $K$ such that if $\|M\| \leq K n^{1/2}$,
		$$
		\P(LCD_{\alpha, \gamma}(v) < \exp(c_{\ref{randomnormal}}n)) \leq \exp(-c_{\ref{randomnormal}}' n)
		$$
		for $\alpha = \beta n^{1/2}$ for any constant $\beta < \lambda$.
	\end{lemma}   
	\begin{proof}
		\begin{align*}
		\P(\exists v \in \mathcal{S}_\C^{n-1}, LCD_{\alpha, \gamma}(v) < e^{c n} \text{ and } M'_n v = 0) &\leq \P(\exists v \in Comp, M'_n v = 0) \\
		& + \P(\exists v \in Incomp, LCD_{\alpha, \gamma}(v)< e^{cn} \text{ and } M'_n v = 0).
		\end{align*}
		Lemma \ref{Compressible} handles the first summand (with a slight adjustment since we are now considering $n-1 \times n$ matrices), giving an upperbound of $\exp(-c'_{\ref{Compressible}} n)$.  Note that $Z_j^T v = 0$ is equivalent to $[v](\hat{Z}_j)^T  = 0$.  By Lemma \ref{lemma:operatornorm}, there exists a $K'$ only depending on $B$ such that
		$$
		\P(\|M'_n v\| \geq (K+ K') n^{1/2}) \leq 2 e^{-n}.
		$$
		We can now choose a small enough constant $c'_{\ref{randomnormal}}$ so that $$2e^{-n} + \exp(-c'_{\ref{Compressible}}) + \exp(-2 c'_{\ref{randomnormal}} n) < \exp(-c'_{\ref{randomnormal}} n).$$  Thus, to complete our argument, it suffices to show that the event
		$$
		\mathcal{E} := \{\exists v \in S_D: M'_n v = 0 \text{ and } \|M'_n\| < (K+K') \sqrt n\}
		$$
		holds with probability at most $\exp(-2c'_{\ref{randomnormal}}n)$ for $ \lambda n^{1/2} \leq D \leq\exp(c_{\ref{randomnormal}} n)$.  A simple union bound over a logarithmic number of disjoint $S_D$ yields the result.  Therefore, for the remainder of the proof, we focus on demonstrating that the probability for the event $\mathcal{E}$ is small for $S_D$ with a fixed $D$ in the range $[\lambda n^{1/2}, \exp(c_{\ref{randomnormal}} n)]$.   
		
		Assume that the event $\mathcal{E}$ holds.  Choose $c_{\ref{randomnormal}}$ so that 
		\begin{equation}\label{epscondition}
		2(K+K') \beta n^{1/2} \exp(c \beta^2 n) \geq \exp(c_{\ref{randomnormal}} n).
		\end{equation}  
		This condition ensures that condition (\ref{eq_epscondition}) is met for our later application of Lemma \ref{singlevector}.  For $ \lambda n^{1/2} \leq D \leq\exp(c_{\ref{randomnormal}} n)$, let $\mathcal{N}$ be the $2\alpha/D$-net for $S_D$ as provided by Lemma \ref{net}.  Choose $y \in \mathcal{N}$ such that $\|v - y\|_2 < 2 \alpha/D$.  By the triangle inequality,
		$$
		\|M'_n y\|_2 \leq \|M'_n\| \|v-y\|_2  \leq 2(K+K') n^{1/2} \frac{\alpha}{D} = \frac{2(K+K') \beta n}{D}
		$$ 
		recalling that $\alpha = \beta n^{1/2}$.  Set $\eps = 2 (K+K') \beta n^{1/2}/D$.  Finally, applying the union bound, we find
		\begin{align*}
		\P(\mathcal{E}) &\leq \P(\exists y \in \mathcal{N}: \|M'_ny\|_2 \leq \eps \sqrt n) \\
		&\leq |\mathcal{N}| (C_{\ref{singlevector}} \eps)^{2n-2} && \text{by Lemma \ref{singlevector} and inequality \ref{epscondition}}\\
		&\leq C_{\ref{net}} \frac{(4D)^2}{\alpha^2} \left(\frac{10D}{n^{1/2}}\right)^{2n} \left(\frac{2 C_{\ref{singlevector}}(K+K') \beta n^{1/2}}{D}\right)^{2n-2} && \text{by Lemma \ref{net} and $\alpha \leq D$} \\
		&\leq \frac{1600 D^2 C_{\ref{net}} }{ n^2} (20 C_{\ref{singlevector}}(K+K'))^{2n-2} \beta^{2n - 3} \\
		&\leq \exp(-2 c_{\ref{randomnormal}}n)
		\end{align*}
		for a suitably small $\beta$.
	\end{proof}
	At this point, we have all the necessary elements to complete the proof of Theorem \ref{LeastSingular}.
	\subsection{Proof of Theorem \ref{LeastSingular}}
	\begin{proof}
		Recall that 
		$$
		\P(s_n(M_n) \leq \eps n^{-1/2}) \leq \P(\inf_{x \in Comp(\delta, \rho)} \|M_n x\|_2 \leq \eps n^{-1/2}) + \P(\inf_{x \in Incomp(\delta, \rho)} \|M_n x\|_2 \leq \eps n^{-1/2}) .
		$$
		By Lemma \ref{Compressible}, the first term on the right is exponentially small.  By Lemma \ref{distance}, the second term is upper bounded by
		$$
		 \frac{1}{\nu_1 n} \sum_{j=1}^n \P(\dist(X_j, H_j) < \eps) \leq \frac{1}{\nu_1 n}  \sum_{j=1}^n \mathcal{L}(|X_j \cdot Z_j|, \eps)\}
		$$
		where $Z_j$ is a vector normal to $H_j$.  Let $\mathcal{E}_{D_j}$ be the event that $LCD_{\alpha, \gamma}(Z_j) > \exp(c n)$, then
		$$
		\mathcal{L}(|X_j \cdot Z_j|, \eps) \leq \mathcal{L}(|X_j \cdot Z_j|\big| \mathcal{E}_{D_j}, \eps) + \P(\overline{\mathcal{E}_{D_j}}).
		$$
		By Theorem \ref{smallball}, the first term is less than $C_{\ref{smallball}} \eps^2 + \exp(-c_{\ref{smallball}}n)$ and the second term is less than $\exp(-c_{\ref{randomnormal}}n)$ by Lemma \ref{randomnormal}.
	\end{proof}
	Finally, the proof of Theorem \ref{MainNoReal} is a consequence of Theorem \ref{LeastSingular}.
\section{Proof of Theorem \ref{MainNoReal}}
\begin{proof}
    Let $K'$ be the constant from Lemma \ref{lemma:operatornorm}.  Set $K = 2K'$.  
    We can choose an $\varepsilon n^{-1/2}$ net of the interval $[-K \sqrt n, K \sqrt n]$ of size at most $2Kn / \varepsilon$.  A real eigenvalue, $\lambda$, in the interval $[-K \sqrt n,K \sqrt n]$ would imply that $s_n(N_n - \lambda_0) \leq \eps n^{-1/2}$ for some $\lambda_0$ in the net.  By Theorem \ref{LeastSingular}, this happens with probability at most $C_{\ref{LeastSingular}} \varepsilon^2 + c_{\ref{LeastSingular}}^{n}$.  Thus, by the union bound, the probability that there exists a real eigenvalue is bounded by $(2K n/\varepsilon) (C_{\ref{LeastSingular}}\varepsilon^2 + c_{\ref{LeastSingular}}^{n})$.  Letting $\varepsilon = c_{\ref{MainNoReal}}^{n}$ with $c_{\ref{MainNoReal}} \in (c_{\ref{LeastSingular}},1)$ yields the result after a slight adjustment to $c_{\ref{MainNoReal}}$.  
\end{proof}

\section*{Acknowledgements}
The author would like to thank Van Vu for his support and helpful discussions.  The author also thanks Oanh Nguyen and Flor Orosz Hunziker for their careful reading of the preliminary drafts and many helpful comments.  Finally, the author is grateful for the many suggestions of the anonymous referree, in particular for pointing out that the definition of LCD in this paper appeared previously in \cite{rudelson2009smallest}. 
\bibliographystyle{plain}
\bibliography{LeastSingular}

\end{document}